\documentclass[sn-mathphys-num]{sn-jnl}

\usepackage{stmaryrd}
\usepackage{amssymb, amsmath, amsthm}
\usepackage{mathtools}
\usepackage{xcolor}
\usepackage{hyperref}
\usepackage{enumerate}
\usepackage{subfig}
\usepackage{enumitem} 
\usepackage{todonotes}
\usepackage{booktabs}
\usepackage[]{algorithm2e}
\usepackage[capitalize]{cleveref}
\usepackage{placeins}
\crefname{equation}{}{}
\usepackage{verbatim}
\usepackage{bbm}

\newtheorem{thm}{Theorem}[section]
\newtheorem{lem}[thm]{Lemma}
\newtheorem{cor}[thm]{Corollary}

\newtheorem{rmk}{Remark}

\newtheorem{assum}{Assumption}
\numberwithin{equation}{section}
\newcommand{\bel}{\begin{equation} \label}
\newcommand{\ee}{\end{equation}}
\def\beq{\begin{equation}}
\def\eeq{\end{equation}}
\newcommand{\jump}[1]{\llbracket#1\rrbracket}
\newcommand{\bea}{\begin{eqnarray}}
\newcommand{\eea}{\end{eqnarray}}
\newcommand{\beas}{\begin{eqnarray*}}
\newcommand{\eeas}{\end{eqnarray*}}

\newcommand{\cL}{\mathcal{L}}

\renewcommand{\div}{\mathrm{div}\,}  

\allowdisplaybreaks
\renewcommand{\theenumi}{\arabic{enumi}}

\def\epsilon{\varepsilon}

\DeclareMathOperator{\dis}{dist}

\newcommand{\wop}{\square_c}

\newcommand{\tnorm}[1]{\vert\hspace{-0.3mm}\Vert#1\Vert\hspace{-0.3mm}\vert}

\providecommand{\norm}[1]{\left\lVert#1\right\rVert}

\renewcommand{\leq}{\leqslant}

\providecommand{\norm}[1]{\left\lVert#1\right\rVert}

\newcommand{\dT}{\mathrm{d}t}
\newcommand{\dX}{\mathrm{d}x}
\newcommand{\dS}{\mathrm{d}S}

\newcommand{\FullyDiscrSpace}[2]{ W^{ {#1},{#2}}_{h, \Delta t  } }

\newcommand{\ProdFullyDiscrSpace}[2]{ \mathcal{W}^{ {#1},{#2}}_{h, \Delta t  } }

\DeclarePairedDelimiterX{\inp}[2]{(}{)}{#1, #2}
\newcommand{\tangular}[1]{ \llbracket\kern-0.5ex|#1|\kern-0.5ex\rrbracket} 

\newboolean{includeextras}
\ifdefined\withextras
\setboolean{includeextras}{true}
\else
\setboolean{includeextras}{false}
\fi

\newcommand{\putextra}[1]{\ifthenelse{\boolean{includeextras}}{#1}{}}

\newcommand{\Uh}{\underline{\mathbf{U}}_h}
\newcommand{\Vh}{\underline{\mathbf{V}}_h}
\newcommand{\Yh}{\underline{\mathbf{Y}}_h}
\newcommand{\Zh}{\underline{\mathbf{Z}}_h}
\newcommand{\Wh}{\underline{\mathbf{W}}_h}

\newcommand{\ul}{\underline{u}}
\newcommand{\yl}{\underline{y}}
\newcommand{\zl}{\underline{z}}
\newcommand{\wl}{\underline{w}}

\newcommand{\Sud}{S^{\uparrow \downarrow}_{\Delta t}}

\newcommand{\dt}{\partial_t}

\newcommand{\changed}[1]{{\color{blue!70!black} #1}}
\newcommand{\jp}[1]{{\color{blue!70!black} #1}}
\renewcommand{\changed}[1]{{\color{black} #1}}
\renewcommand{\jp}[1]{{\color{black} #1}}

\usepackage{tikz}
\usepackage{pgfplots,pgfplotstable}
\usepgfplotslibrary{colorbrewer,groupplots}
\usepackage{caption}
\pgfplotsset{compat=1.18}

\pgfplotsset{
cycle list/Set1-5,
cycle multiindex* list={
mark list*\nextlist
Set1-5\nextlist
},
}

\pgfplotsset{
    discard if not/.style 2 args={
        x filter/.append code={
            \edef\tempa{\thisrow{#1}}
            \edef\tempb{#2}
            \ifx\tempa\tempb
            \else
                
            \fi
        }
    }
}

\providecommand{\jp}[1]{\newline {\color{cyan} \hspace*{-0.5cm} $\rightsquigarrow$ \underline{\textbf{JP:}} \textit{#1} \newline}}

\date{\today}
\begin{document}

\title{Variational data assimilation for the wave equation in heterogeneous media}
\subtitle{Numerical investigation of stability}

\author[1]{\fnm{Erik} \sur{Burman}}\email{e.burman@ucl.ac.uk}
\author[2]{\fnm{Janosch} \sur{Preuss}}\email{janosch.preuss@inria.fr}
\author[3]{\fnm{Tim} \sur{van Beeck}}\email{t.beeck@math.uni-goettingen.de}

\affil[1]{\orgdiv{Department of Mathematics}, \orgname{University College London}, \orgaddress{\street{Gower Street}, \city{London}, \postcode{WC1E 6BT}, \country{United Kingdom}}}
 
 \affil[2]{\orgdiv{Inria Project-Team Makutu}, \orgname{Université de Pau et des Pays de l’Adour}, \orgaddress{\street{Avenue de l'Université - BP 1155}, \city{Pau}, \postcode{64013}, \country{France}}}

\affil[3]{\orgdiv{Institute for Numerical and Applied Mathematics}, \orgname{University of G\"{o}ttingen}, \orgaddress{\street{Lotzestr. 16-18}, \city{G\"{o}ttingen}, \postcode{37083}, \country{Germany}}}

\abstract{In recent years, several numerical methods for solving the unique continuation problem for the wave equation in a homogeneous medium with given data on the lateral boundary of the space-time cylinder have been proposed. This problem enjoys Lipschitz stability if the geometric control condition is fulfilled, which allows devising optimally convergent numerical methods. In this article, we investigate whether these results carry over to the case in which the medium exhibits a jump discontinuity. Our numerical experiments suggest a positive answer. However, we also observe that the presence of discontinuities in the medium renders the computations far more demanding than in the homogeneous case.
}


\maketitle

\section{Introduction}

We consider the discretization of a data assimilation problem for the wave equation in heterogeneous media. In the following, let $\Omega \subset \mathbb{R}^d$, $d \in \{1,2,3\}$, be a bounded domain with smooth boundary $\partial \Omega$. We assume that $\Omega$ is split by a smooth $(d-1)$-dimensional interface $\Gamma$ into two disjoint components such that $\Omega \setminus \Gamma = \Omega_1 \cup 
\Omega_2$ and $\Omega_1 \cap \Omega_2 = \emptyset$.
 
We consider a scalar wave speed
\[
 c(x) := \begin{cases}
       c_1(x),  & x \in \Omega_1, \\
        c_2(x), & x \in \Omega_2,
    \end{cases}
    \quad 
    \text{ with } c_i \in C^\infty(\Omega_i), \; i = 1,2,
\]
which may be discontinuous across the interface $\Gamma$, but is constant in time. For a final time $T > 0$, we consider the space-time domain $Q = (0,T) \times \Omega$ with lateral boundary $\Sigma \coloneqq (0,T) \times \partial \Omega$, and for any nonempty subset $\omega \subset \Omega$, we set $\omega_T \coloneqq (0,T) \times \omega$. 

Let $\wop$ be the wave operator with discontinuous wave-speed defined as 
\begin{equation*}
    \wop u := \partial_t^2 u - \div(c^2 \nabla u).
\end{equation*}
In particular, we write $\square_1$ if $c_i = 1$ for $i = 1,2$, corresponding to the case of a homogeneous medium with unit wave speed.

Then, we consider the problem: find $u : Q \rightarrow \mathbb{R}$ such that 
\begin{equation}\label{eq:waveEquation}
        (\wop u, u) = (0,0) \quad \text{ in } Q \times \Sigma,
\end{equation}
along with the natural interface conditions 
\begin{equation}\label{eq:IF-conditions}
	\jump{u}_{\Gamma} = 0 \text{ on } \Sigma, \qquad 
	\quad \jump{c^2 \nabla u}_{\Gamma} \cdot \mathbf{n}_{\Gamma} = 0 \text{ on } \Sigma, 
\end{equation}
where $\mathbf{n}_{\Gamma}$ is the normal vector on $\Gamma$ pointing into $\Omega_+$ 
and the jump across $\Gamma$ is defined\textsuperscript{1}\footnotemark \footnotetext{\textsuperscript{1}And analogously for vector-valued functions.} as $\jump{u}_{\Gamma} = u_{+}-u_{-}$ with $u_{\pm}$ being the traces 
of $u|_{\Omega_{\pm}}$ on $\Gamma$. \par 
To ensure that the problem is well posed, one usually prescribes initial data
\begin{equation}\label{eq:initialData}
    (u,\dt u) \vert_{t = 0} = (u_0,u_1) \text{ in } \Omega, \tag{IVP} 
\end{equation}
where $u_0 \in H^1(Q)$ and $u_1 \in L^2(Q)$ are given. For a detailed analysis of the well-posedness of the problem \eqref{eq:waveEquation}-\eqref{eq:IF-conditions}+\eqref{eq:initialData} in heterogeneous media, we refer to \cite{StolkPhD}. 

In this work, we consider a \emph{data assimilation} problem instead, where the initial data is unknown, and we assume that there are given measurements $u_{\omega} \in L^2(\omega_T)$ prescribed in a data domain $\omega_T$, with $\omega \subset \Omega$, $\omega \not = \emptyset$:
\begin{equation}\label{eq:dataMatch}
    u = u_{\omega} \text{ in } \omega_T. \tag{DA}
\end{equation}
In a medium without an interface, the problem \eqref{eq:waveEquation}+\eqref{eq:dataMatch} is well-studied; see for example \cite{CM15,BFO20,BFMO21,MM21,DMS23,BDE24}. In particular, if the boundary $\partial \Omega$ is strictly convex, the problem is Lipschitz stable provided that the set $\omega_T \subset Q$ fulfills the geometric control condition (GCC) \cite{BLR92}. 
Roughly speaking, a set $\omega \subset \Omega$ and a final time $T$ fulfill the GCC if all geometric optic rays traveling with unit speed\textsuperscript{2}\footnotemark \footnotetext{\textsuperscript{2}We consider the metric to be adapted to the wave speed $c$.} in $\Omega$, accounting for possible reflections on the boundary $\partial \Omega$, intersect the set $(0,T) \times \omega$. For a more precise definition, we refer to \cite{BLR92}. Assuming that the GCC is satisfied, we have the following result.

\begin{thm}\label{thm:Lipschitz}
    Assume that $\partial \Omega$ is strictly convex and $c_i = 1$ for $i = 1,2$. If $\omega_T \subset Q$ fulfills the GCC, then there exists a constant $C > 0$ such that for any $\phi \in H^1(Q)$, we have the following estimates:
    \begin{align*}
        \Vert \phi \Vert_{L^\infty(0,T;L^2(\Omega))} + \Vert \dt \phi \Vert_{L^2(0,T;H^{-1}(\Omega))} &\le C \left(  \Vert \phi \Vert_{L^2(\omega_T)} + \Vert \phi \Vert_{L^2(\Sigma)} + \Vert \square_1 \phi \Vert_{H^{-1}(Q)} \right), \\
        \Vert \phi \vert_{t = 0} \Vert_{L^2(\Omega)} + \Vert \dt \phi \vert_{t = 0} \Vert_{H^{-1}(\Omega)} &\le C \left(\Vert \phi \Vert_{L^2(\omega_T)} + \Vert \phi \Vert_{L^2(\Sigma)} + \Vert \square_1 \phi \Vert_{H^{-1}(Q)} \right). 
    \end{align*}
\end{thm}

\begin{proof}
    See \cite[Thm. A.4]{BFMO21control} and \cite[Rem. A.5]{BFMO21control}.
\end{proof}

In the heterogeneous case, where $c_1 \not =  c_2$, the problem is less studied. In the absence of the GCC, we have to assume that the final time $T$ is large enough to ensure that problem \eqref{eq:waveEquation}-\eqref{eq:IF-conditions}+\eqref{eq:dataMatch} is well-posed.
To be precise, 
\jp{Thm. 5.19 of \cite{Filippas22} (restated below as \Cref{thm:filippas}) guarantees that solutions of \eqref{eq:waveEquation}-\eqref{eq:IF-conditions}+\eqref{eq:dataMatch} 
are unique if} 
\begin{equation}\label{eq:Tcondition}
	T > 2 \sup_{x \in \Omega} \operatorname{dist}_{c}(x, \jp{\omega}),
\end{equation}
where $\operatorname{dist}_c(x, \jp{\omega}) := \inf_{y \in \jp{\omega} } d_c(x,y)$. Here, $d_c$ is the distance function defined as the length of the shortest continuous path in $\Omega$ connecting $x$ and $y$ with respect to the metric adapted to the discontinuous wave speed $c$.
\jp{We remark that the condition on the time in \cref{eq:Tcondition} is identical to the one assumed in the classical unique continuation result of Tataru for a smooth coefficient $c$.}
While the condition \eqref{eq:Tcondition} guarantees unique solvability, the stability of problem \eqref{eq:waveEquation}-\eqref{eq:IF-conditions}+\eqref{eq:dataMatch} can be poor. However, as shown in \cite[Thm. 5.19]{Filippas22}, Assumption \eqref{eq:Tcondition} grants the following stability result holds. 
\begin{thm}[Thm. 5.19 of \cite{Filippas22}]\label{thm:filippas}
Let $\jp{\omega} \subset \Omega$ be nonempty and let $T > 0$ be s.t. \eqref{eq:Tcondition} holds. Then, there exists $C,\kappa$ such that for any initial data $(u_0,u_1) \in H^1_0(\Omega) \times L^2(\Omega)$ and $u$ solving \eqref{eq:IF-conditions}+\eqref{eq:initialData} and  
\[
\wop u \in L^2(Q), \quad u = 0 \text{ on } \Sigma,
\]
one has for any $\mu > 0$ that 
    \begin{equation}\label{eq:FilippasEstimate}
        \Vert (u_0,u_1) \Vert_{L^2 \times H^{-1}} \le C e^{\kappa \mu} \left( \Vert u \Vert_{L^2(\omega_T)} + \Vert \wop u \Vert_{L^2(Q)} \right) + \frac{C}{\mu} \Vert (u_0,u_1) \Vert_{H^1 \times L^2}. 
    \end{equation}
    In particular, if $(u_0,u_1) \not = (0,0)$ we have that 
    \begin{equation}\label{eq:FilippasEstimateLog}
        \Vert (u_0,u_1) \Vert_{L^2 \times H^{-1}} \le C \frac{\Vert (u_0,u_1) \Vert_{H^1 \times L^2}}{\log \left( 1 + \frac{\Vert (u_0,u_1) \Vert_{H^1 \times L^2}}{ \Vert u \Vert_{L^2(\omega_T)} + \Vert \wop u \Vert_{L^2(Q)}} \right)}. 
    \end{equation} 
\end{thm} 

In this article, we consider a discontinuous-in-time finite element method to solve the data assimilation problem \eqref{eq:waveEquation}-\eqref{eq:IF-conditions}+\eqref{eq:dataMatch} in the heterogeneous case. This method was originally introduced for the homogeneous case in \cite{BP24}. The focus of this article is not on the numerical analysis of the proposed discretization, but rather on investigating the numerically observed stability under the assumption of the GCC. Specifically, we want to investigate numerically whether the following assumption -- motivated by \cref{thm:Lipschitz} for the homogeneous case -- is justified. 

\begin{assum}\label{assum:LipschitzStability}
   If $ \omega_T \subset Q$ fulfills the GCC, then there exists a constant $C > 0$ such that for any $\phi \in H^1(Q)$ we have the following estimates
    \begin{align*}
        \Vert \phi \Vert_{L^\infty(0,T;L^2(\Omega))} + \Vert \dt \phi \Vert_{L^2(0,T;H^{-1}(\Omega))} &\le C \left(  \Vert \phi \Vert_{L^2(\omega_T)} + \Vert \phi \Vert_{L^2(\Sigma)} + \Vert \square_c \phi \Vert_{H^{-1}(Q)} \right), \\
        \Vert \phi \vert_{t = 0} \Vert_{L^2(\Omega)} + \Vert \dt \phi \vert_{t = 0} \Vert_{H^{-1}(\Omega)} &\le C \left(\Vert \phi \Vert_{L^2(\omega_T)} + \Vert \phi \Vert_{L^2(\Sigma)} + \Vert \square_c \phi \Vert_{H^{-1}(Q)} \right). 
    \end{align*}
\end{assum}

In summary, our goal is to understand whether the Lipschitz stability property is preserved in the presence of strong heterogeneity in the medium, provided that the GCC holds. This question is of high practical importance, as discontinuous media appear in many applications, for example, in seismology. If the corresponding data assimilation problems were indeed subject to the poor logarithmic stability property ensured by \eqref{eq:FilippasEstimateLog}, this would pose an enormous challenge for the design of accurate and reliable numerical methods. 
\begin{rmk}
\jp{While \Cref{assum:LipschitzStability} appears natural, it is possible that in certain configurations some additional conditions on the geometry and minimal time $T$ have to be imposed to prove Lipschitz stability. This is due to the complex interactions of rays with the interface $\Gamma$. For further discussion we refer to the reference \cite{LG23} in which sufficient conditions for boundary controllability of wave equations with a jumping coefficient have been derived  
and in particular to Remark 1.6 therein.} 
\end{rmk}

\section{Fully discrete discontinuous-in-time discretization}\label{sec:discretization} 
In this section, we present a fully discrete discontinuous-in-time finite element method for discretizing the data assimilation problem \eqref{eq:waveEquation}-\eqref{eq:IF-conditions}+\eqref{eq:dataMatch}, extending the method from \cite{BP24} to the heterogeneous case. \Cref{sec:spaceTimeDiscretization} describes the geometric partition of the space-time domain $Q$ into time slabs, followed by the full discretization in \cref{sec:method}. \jp{The chosen discretization allows for an iterative solution of the linear systems which is described in detail in \cite[Section 5]{BP24}. }

\subsection{Discretization of the space-time domain}\label{sec:spaceTimeDiscretization}
Let $\hat{\mathcal{T}}_h$ be a quasi-uniform triangulation with mesh size $h$ such that $\hat{\Omega}:= \cup_{K \in \hat{\mathcal{T}}_h} K $ contains $\Omega$, i.e.\ $\Omega \subset \hat{\Omega}$. Setting $\mathcal{T}_h := \{ K \cap \Omega \mid K \in  \hat{\mathcal{T}}_h \}$ then ensures that $\Omega = \cup_{K \in \mathcal{T}_h} K$ holds true.  
We recall from  \cite[Sec. 4.2]{BFMO21control} that the triangulation $\{ \hat{\mathcal{T}}_h \mid h > 0 \}$ can be constructed such that for $h$ small enough the following continuous trace inequality holds:  For all $v \in [H^1(K)]^d$ and $K \in \mathcal{T}_h$, we have
\begin{equation}\label{eq:traceInequality}
    \Vert v \Vert_{[L^2(\partial K)]^d} \le C \left(h^{-1/2} \Vert v \Vert_{[L^2(K)]^d} + h^{1/2} \Vert \nabla v \Vert_{[L^2(K)]^d} \right).  
\end{equation}
We assume additionally that the triangulation $\mathcal{T}_h$ fits the interface $\Gamma$ and the data domain $\omega$. An extension to allow for unfitted interfaces could be possible by utilizing techniques introduced in \cite{BP25} for a stationary problem, but this is certainly beyond the scope of this article. 

For polynomial degree $k \ge 1$, we define the $H^1$-conforming finite element space 
\begin{equation}
    V_h^k := \{ v \in H^1(\Omega) : v \vert_{K} \in \mathcal{P}^k(K) \ \forall K \in \mathcal{T}_h \},
\end{equation}
where $\mathcal{P}^k(K)$ denotes the space of polynomials of degree at most $k \in \mathbb{N}$ on $K \in \mathcal{T}_h$. Furthermore, we partition the time axis into $N$ subintervals $I_n = (t_n,t_{n+1})$, $n = 0, \dots, N-1$, where $0 = t_0 \le t_1 \le \dots \le t_N = T$. We assume that the intervals are of equal length and denote $\Delta t = \vert t_{n+1} -t_n \vert$. Then, we partition $Q$ and $\Sigma$ into time-slabs 
\begin{equation}
    \begin{aligned}
        Q^n := &I_n \times \Omega, \quad \Sigma^n := I_n \times \Sigma, \quad n = 0, \dots, N-1, \\
        Q = &\bigcup_{n = 0}^{N-1} Q^n, \quad \phantom{:} \Sigma = \bigcup_{n = 0}^{N-1} \Sigma^n.
    \end{aligned}
\end{equation}
In the following, we denote the space-time integrals on the time slabs as 
\begin{equation*}
    (u,v)_{Q^n} := \int_{I_n} \int_{\Omega} uv \ \dX \dT, \quad (u,v)_{\Sigma^n} := \int_{I_n} \int_{\Sigma} uv \ \dS \dT,
\end{equation*}
and define $\Vert v \Vert^2_{Q^n} := (v,v)_{Q^n}$ and $\Vert v \Vert^2_{\Sigma^n} := (v,v)_{\Sigma^n}$. 
Furthermore, we set 
\begin{equation*}
    \omega^n := I_n \times \omega, \quad (u,v)_{\omega^n} := \int_{I_n} \int_{\omega} uv \ \dX \dT, \quad \Vert v \Vert^2_{\omega_T} := \sum_{n = 0}^{N-1} (v,v)_{\omega^n}.
\end{equation*}
Finally, we define the time jump operator 
\begin{equation*}
    v^n_{\pm} (x) := \lim_{s \rightarrow 0^+} v(x,t_n \pm s), \quad \jump{v^n} := v^n_+ - v^n_-. 
\end{equation*}

\subsection{Full discretization}\label{sec:method}
 We define the discontinuous in time finite element spaces 
\begin{equation}
    \FullyDiscrSpace{k}{q} := \otimes_{n = 0}^{N-1} \mathcal{P}^q(I_n) \otimes V_h^k, \quad \ProdFullyDiscrSpace{k}{q} :=  \FullyDiscrSpace{k}{q} \times \FullyDiscrSpace{k}{q}, \quad q \in \mathbb{N}_0, k \in \mathbb{N}. 
\end{equation}
For elements of $\ProdFullyDiscrSpace{k}{q}$, we use the notation $\Uh = (\ul_1,\ul_2) \in \ProdFullyDiscrSpace{k}{q}$. In the following, we denote
\begin{equation*}
    a(u,v)_{Q^n} := \int_{I_n} \int_{\Omega} c^2 \nabla u \cdot \nabla v \ \dX \dT,
\end{equation*}
and introduce a bilinear form $A$ that represents the wave-equation \eqref{eq:waveEquation} in mixed formulation:  
\begin{equation}
    \begin{aligned}
        A[\Uh,\Yh] := \sum_{n = 0}^{N -1} \Big\{ &(\dt \ul_2, \yl_1)_{Q^n} + a(\ul_1,\yl_1)_{Q^n} + (\dt \ul_1 - \ul_2,\yl_2)_{Q^n} \\
        &- (c^2 \nabla \ul_1 \cdot \mathbf{n}, \yl_1)_{\Sigma^n} \Big\},
    \end{aligned}
\end{equation} 
where $\Uh \in \ProdFullyDiscrSpace{k}{q}$ and $\Yh \in \ProdFullyDiscrSpace{k^\ast}{q^\ast}$ with $k, k^\ast, q \in \mathbb{N}$ and $q^\ast \in \mathbb{N}_0$. 
To approximate solutions of \eqref{eq:waveEquation}-\eqref{eq:IF-conditions}+\eqref{eq:dataMatch}, we search for stationary points of the Lagrangian 
\begin{align*}
    \cL_h (\Uh, \Zh) := &\frac{1}{2} \Vert \ul_1 - u_{\omega} \Vert^2_{\omega_T} + A[\Uh,\Zh] + \frac{1}{2}  S_h(\Uh,\Uh) \\
    &- \frac{1}{2}  S_h^\ast(\Zh,\Zh) + \frac{1}{2} \Sud(\Uh,\Uh),
\end{align*}
\changed{for $\Uh \in \ProdFullyDiscrSpace{k}{q}, \Zh \in \ProdFullyDiscrSpace{k^\ast}{q^\ast}$}, where $S_h$, $S_h^\ast$, and $\Sud$ are stabilization terms yet to be defined. Note that the first and the second terms of $\cL_h$ incorporate the data and the PDE constraints, respectively. To define the stabilization terms $S_h$ and $S_h^\ast$, we introduce the following terms: 
\begin{align*}
    J(\Uh,\Wh) &:= \sum_{n = 0}^{N -1} \int_{I_n} \sum_{F \in \mathcal{F}_i} h (\jump{c^2 \nabla \ul_1}  \cdot \mathbf{n}, \jump{c^2 \nabla \wl_1}  \cdot \mathbf{n})_F \ \dT, \\
    G(\Uh,\Wh) &:= \sum_{n = 0}^{N -1} \int_{I_n} \sum_{K \in \mathcal{T}_h} h^2 (\dt \ul_2 - \div (c^2 \nabla \ul_1),\dt \wl_2 -\div( c^2 \nabla \wl_1))_K \ \dT, \\
    I_0(\Uh,\Wh) &:= \sum_{n = 0}^{N -1} (\ul_2 - \dt \ul_1, \wl_2-\dt \wl_1)_{Q_n}, \qquad 
    R(\Uh,\Wh) := \sum_{n = 0}^{N -1} h^{-1} (\ul_1,\wl_1)_{\Sigma^n},
\end{align*}
where $\jump{c^2 \nabla \cdot} \cdot \mathbf{n}$  denotes the jump of the normal derivative over interior facets $\mathcal{F}_i$.
$J(\cdot,\cdot)$ is a continuous interior penalty term in space, $G(\cdot,\cdot)$ is a Galerkin least squares term enforcing the PDE locally on each element, $I_0(\cdot,\cdot)$ enforces that $\ul_2 = \dt \ul_1$, and $R(\cdot,\cdot)$ ensures boundary stability. 
Then we define the primal stabilizer $S_h$ as 
\begin{equation}\label{eq:primalStab}
    S_h(\Uh,\Wh) := J(\Uh,\Wh) + I_0(\Uh,\Wh) + G(\Uh,\Wh) + R(\Uh,\Wh), 
\end{equation}
and the dual stabilizer $S_h^\ast$ through 
\begin{equation}
    S_h^\ast(\Yh,\Zh) := \sum_{n = 0}^{N-1} \left\{ (\yl_1,\zl_1)_{Q^n} + a(\yl_1,\zl_1)_{Q^n} + (\yl_2,\zl_2)_{Q^n} + h^{-1} (\yl_1,\zl_1)_{\Sigma^n} \right\}.
\end{equation}
The remaining stabilization term $\Sud$ imposes regularity on the discontinuities in time and is defined as 
\begin{equation}
    \Sud (\Uh,\Wh) := \underline{I}_1^{\uparrow \downarrow}(\Uh,\Wh) + \underline{I}_2^{\uparrow \downarrow}(\Uh,\Wh),
\end{equation}
where 
\begin{align*}
    \underline{I}_1^{\uparrow \downarrow}(\Uh,\Wh) &:= \sum_{n = 1}^{N-1} \left\{ \frac{1}{\Delta t} (\jump{\ul_1^n},\jump{\wl_1^n})_{\Omega} + \Delta t (c^2 \jump{\nabla \ul_1^n},c^2 \jump{\nabla \wl_1^n})_{\Omega}\right\}, \\
    \underline{I}_2^{\uparrow \downarrow}(\Uh,\Wh) &:= \sum_{n = 1}^{N-1} \frac{1}{\Delta t} (\jump{\ul_2^n},\jump{\wl_2^n})_{\Omega}.
\end{align*}

With the definition of $\cL_h$, the first order optimality conditions take the form: Find $(\Uh,\Zh) \in \ProdFullyDiscrSpace{k}{q} \times \ProdFullyDiscrSpace{k^\ast}{q^\ast}$ such that 
\begin{alignat*}{2}
    (\ul_1,\wl_1)_{\omega_T} \! + \! A[\Wh,\Zh] + S_h(\Uh,\Wh) + \Sud(\Uh,\Wh) &= (u_{\omega},\wl_1)_{\omega_T} \ &&\forall \Wh \in \ProdFullyDiscrSpace{k}{q}, \\
    A[\Uh,\Yh] -  S_h^\ast(\Yh,\Zh) &= 0 \ &&\forall \Yh \in \ProdFullyDiscrSpace{k^\ast}{q^\ast}.
\end{alignat*}
In a more compact form, we can write these conditions as: Find $(\Uh,\Zh) \in \ProdFullyDiscrSpace{k}{q} \times \ProdFullyDiscrSpace{k^\ast}{q^\ast}$ such that
\begin{equation}\label{eq:discreteProblem}
    B[(\Uh,\Zh),(\Wh,\Yh)] = (u_{\omega},\wl_1)_{\omega_T} \quad \forall (\Wh,\Yh) \in \ProdFullyDiscrSpace{k}{q} \times \ProdFullyDiscrSpace{k^\ast}{q^\ast},
\end{equation}
where 
\begin{equation}
    \begin{aligned}
        B[(\Uh,\Zh),(\Wh,\Yh)] \coloneqq \ &(\ul_1,\wl_1)_{\omega_T} + A[\Wh,\Zh]+ S_h(\Uh,\Wh) \\
        &+  \Sud(\Uh,\Wh) + A[\Uh,\Yh] - S_h^\ast(\Yh,\Zh).
    \end{aligned}
\end{equation}

\section{Error analysis}
This section outlines how the steps from \cite{BP24} can be adapted to the heterogeneous setting to analyze the approximation of \eqref{eq:waveEquation}-\eqref{eq:IF-conditions}+\eqref{eq:dataMatch} through \eqref{eq:discreteProblem}. In \cref{sec:analysis:infsup}, we establish inf-sup stability and continuity of the discrete problem with respect to suitable discrete norms.
In \cref{sec:analysis:approx}, we show that the error between the first component $\ul_1$ of the discrete solution and the continuous solution $u$ is bounded by the best approximation error in these norms. Finally, we argue that \cref{assum:LipschitzStability} allows to recover convergence rates in a physically relevant norm. 

\subsection{Inf-sup stability and continuity}\label{sec:analysis:infsup}
\changed{In the following, we assume that \eqref{eq:Tcondition} is satisfied.}
Let $\vert \cdot \vert_{S_h}, \vert \cdot \vert_{\uparrow \downarrow}$, and $\Vert \cdot \Vert_{S_h^{\ast}}$ be the (semi-)norms induced by the stabilizers $S_h$, $\Sud$, and $S_h^{\ast}$:
\begin{equation}
    \vert \Uh \vert^2_{S_h} := S_h(\Uh,\Uh), \quad \vert \Uh \vert^2_{\uparrow \downarrow} := \Sud(\Uh,\Uh), \quad \Vert \Zh \Vert^2_{S_h^{\ast}} := S_h^{\ast}(\Zh,\Zh).
\end{equation}
Then, we define the discrete norm 
\begin{equation}
    \tnorm{ (\Uh,\Zh) }^2 :=  \vert \Uh \vert^2_{S_h} + \vert \Uh \vert^2_{\uparrow \downarrow} + \Vert \ul_1 \Vert^2_{\omega_T} + \Vert \Zh \Vert^2_{S_h^\ast},
\end{equation}
and its strengthened version
\begin{equation}\label{eq:tnormwop}
    \begin{aligned}
        \tnorm{(\Uh, \Zh)}^2_{\wop} := \tnorm{(\Uh, \Zh)}^2 &+ \sum_{n = 0}^{N -1} \Big\{ \Vert \dt \ul_2 \Vert^2_{Q^n} + \Vert c^2 \nabla \ul_1 \Vert^2_{Q^n} + \Vert \dt \ul_1 \Vert^2_{Q^n} \\
        &+ \Vert \ul_2 \Vert^2_{Q^n} + \int_{I_n} \sum_{K \in \mathcal{T}_h} h^2 \Vert c^2 \ul_1 \Vert^2_{H^2(K)} \dT \Big\}. 
    \end{aligned}
\end{equation}

\begin{lem}
    The expressions $\tnorm{(\cdot,\cdot)}$ and $\tnorm{(\cdot,\cdot)}_{\wop}$ are norms on $\ProdFullyDiscrSpace{k}{q} \times \ProdFullyDiscrSpace{k^\ast}{q^\ast}$.  
\end{lem}

\begin{proof}
    It suffices to show that $\tnorm{(\cdot,\cdot)}$ is a norm. Since $\tnorm{(\cdot,\cdot)}$ is a semi-norm by definition, we only require that $\tnorm{(\Uh,\Zh)} = 0$ implies $\Uh=\Zh= 0$. Assume that $\tnorm{(\Uh,\Zh)} = 0$. Then, in particular, $\vert \Uh \vert_{\uparrow \downarrow} = 0$ and therefore $(\ul_1,\ul_2) \in [H^1(Q)]^2$. Furthermore, by definition of the stabilization terms $S_h$ and $S_h^{\ast}$, we have that $\Zh = 0$, $\ul_1 \vert_{\Sigma} = 0$, $\ul_1 \vert_{\omega_T} = 0$ and $\dt \ul_1 = \ul_2$. Similar to \cite[Lemma 2]{BP24}, partial integration yields 
    \begin{align*}
	   & \Vert \wop \ul_1 \Vert^2_{H^{-1}(Q)} := \sup_{\substack{  y \in H^1_0(Q), \\ \norm{y}_{H^1(Q) } = 1  }} \int_{Q} \left\{ -(\dt \ul_1) \dt y + c^2 \nabla \ul_1 \nabla y \right\}  \\
        &= \sup_{\substack{  y \in H^1_0(Q), \\ \norm{y}_{H^1(Q) } = 1  }} \Big\{ \sum_{n = 0}^{N-1} (\ul_2 - \dt \ul_1, \dt y)_{Q_n} 
        + \sum_{n = 0}^{N-1} \int_{I_n} \sum_{K \in \mathcal{T}_h} (\dt \ul_2 - \div(c^2 \nabla \ul_1),y)_{K} \ \dT \\
	    & \hspace{6em} + \sum_{n = 0}^{N-1} \int_{I_n} \sum_{F \in \mathcal{F}_i} (\jump{c^2 \nabla \ul_1} \cdot \mathbf{n}, y)_F \ \dT + \sum_{n = 0}^{N-1} (\jump{\ul_1^n},y)_{\Sigma^n} \Big\},
    \end{align*}
It follows by definition of the stabilizers and $\tnorm{(\Uh,\Zh)} = 0$ that $\Vert \wop \ul_1 \Vert_{H^{-1}(Q)} = 0$ and that the interface conditions \eqref{eq:IF-conditions} are fulfilled. Thus, by taking $\mu \rightarrow \infty$ in \cref{thm:filippas}, we conclude that $u_0 = u_1 = 0$. As solutions to \eqref{eq:waveEquation}-\eqref{eq:IF-conditions}+\eqref{eq:initialData} depend continuously on the data \cite{StolkPhD}, it follows that $\Uh = 0$. 
\end{proof}

 From this result and the identity
\[
B[(\Uh,\Zh),(\Uh,-\Zh)] = \tnorm{(\Uh,\Zh)}^2 = \tnorm{(\Uh,\Zh)}\tnorm{(\Uh,-\Zh)},
\]
it directly follows that the bilinear form $B$ enjoys inf-sup stability on $\ProdFullyDiscrSpace{k}{q} \times \ProdFullyDiscrSpace{k^\ast}{q^\ast}$ with respect to the $\tnorm{(\cdot,\cdot)}$-norm. 

\begin{cor}\label{cor:infsup}
    There exists a constant $C_B>0$ such that 
    \begin{equation}
        \sup_{(\Wh,\Yh) \in \ProdFullyDiscrSpace{k}{q} \times \ProdFullyDiscrSpace{k^\ast}{q^\ast}} \frac{B[(\Uh,\Zh),(\Wh,\Yh)]}{\tnorm{(\Wh,\Yh)}} \ge C_B \tnorm{(\Uh,\Zh)}
    \end{equation}
    for all $(\Uh,\Zh) \in \ProdFullyDiscrSpace{k}{q} \times \ProdFullyDiscrSpace{k^\ast}{q^\ast}$.
\end{cor}


The trace inequality \eqref{eq:traceInequality} and the Cauchy-Schwarz inequality immediately yield the continuity of the bilinear form $A$.

\begin{lem}[Continuity of $A$]\label{lem:continuityA}
If $\mathbf{U}|_{Q^n} \in [H^1(Q^n) \cap L^2(0,T;H^2(\mathcal{T}_h))] \times H^1(Q^n)$ for $n=0,\ldots,N-1$ and for all $\Yh \in \ProdFullyDiscrSpace{k^\ast}{q^\ast}$, we have that 
    \begin{equation*}
        A[\mathbf{U},\Yh] \le C \tnorm{(\mathbf{U}, 0)}_{\wop} \tnorm{(0,\Yh)}. 
    \end{equation*} 
\end{lem}

\begin{proof}
    \changed{
    First, note that $\tnorm{(0,\Yh)} = \Vert \Yh \Vert_{S_h^\ast}$. With the Cauchy-Schwarz inequality, we obtain
    \begin{equation*}
        \sum_{n = 0}^{N-1} \left\{ (\dt u_2, \underline{y}_1)_{Q^n} + a(u_1,\underline{y}_1)_{Q^n} \right\} \le C \left( \sum_{n = 0}^{N-1} \Vert \dt u_2 \Vert^2_{Q^n} + \Vert c^2 \nabla u_1 \Vert^2_{Q^n} \right)^{1/2} \! \! \! \Vert \Yh \Vert_{S_h^{\ast}},
    \end{equation*}
    and 
    \begin{equation*}
        \sum_{n = 0}^{N-1} (\dt u_1 - u_2,\underline{y}_2)_{Q^n} \le C \left( \sum_{n = 0}^{N-1} \Vert \dt u_1 \Vert^2_{Q^n} + \Vert u_2 \Vert^2_{Q^n} \right)^{1/2} \Vert \Yh \Vert_{S_h^{\ast}}.
    \end{equation*}
    Since the interface $\Gamma$ is fitted by the triangulation, we have that $c^2 \nabla u_1 \in H^1(K)$ for each $K \in \mathcal{T}_h$. Thus, we can apply the trace inequality \eqref{eq:traceInequality} to obtain
    \begin{align*}
        \sum_{n = 0}^{N-1} (c^2 \nabla &u_1 \cdot n, \underline{y}_1)_{\Sigma^n} \\
        &\le \left( \sum_{n = 0}^{N-1} \int_{I_n} \sum_{K \in \mathcal{T}_h} h \Vert c^2 \nabla v \Vert_{\partial \Omega \cap \partial K}^2 \dT \right)^{1/2} \left( \sum_{n = 0}^{N-1} \int_{I_n} \sum_{K \in \mathcal{T}_h} h^{-1} \Vert \underline{y}_1 \Vert_{\partial \Omega \cap \partial K}^2 \dT \right)^{1/2} \\
        &\le C \left( \sum_{n = 0}^{N-1} \left\{ \Vert c^2 \nabla u_1 \Vert^2_{Q^n} + \int_{I_n} \sum_{K \in \mathcal{T}_h} h^2 \Vert c^2 u_1 \Vert_{H^2(K)} \dT \right\} \right)^{1/2} \Vert \Yh \Vert_{S_h^{\ast}}.
    \end{align*}
    The claim follows from putting all estimates together.}
\end{proof}

\subsection{Approximation results}\label{sec:analysis:approx}
With \cref{cor:infsup} and \cref{lem:continuityA}, we can show that the approximation error in the $\tnorm{(\cdot,\cdot)}$-norm is bounded by the best approximation error in the $\tnorm{(\cdot,\cdot)}_{\wop}$-norm.

\begin{thm}\label{thm:bestapprox}
    Let $u$ be a sufficient regular solution of \eqref{eq:waveEquation}-\eqref{eq:IF-conditions}+\eqref{eq:dataMatch} and $(\Uh,\Zh) \in \ProdFullyDiscrSpace{k}{q} \times \ProdFullyDiscrSpace{k^\ast}{q^\ast}$ be the solution to \eqref{eq:discreteProblem}. Set $\mathbf{U} := (u,\partial_t u)$. Then, there exists a constant $C>0$ such that
    \begin{equation}
   \tnorm{(\mathbf{U} - \Uh,\Zh)} \le \left( 1 + \frac{C}{C_B} \right) \inf_{\Vh \in \ProdFullyDiscrSpace{k}{q}} \tnorm{(\mathbf{U} - \Vh, 0)}_{\wop}.  
    \end{equation}
\end{thm}

\begin{proof}
    Let $\Vh \in \ProdFullyDiscrSpace{k}{q}$  and $(\Wh,\Yh) \in \ProdFullyDiscrSpace{k}{q} \times \ProdFullyDiscrSpace{k^\ast}{q^\ast}$ be arbitrary. The triangle inequality yields 
    \begin{equation*}
        \tnorm{(\mathbf{U} - \Uh,\Zh)} \le \tnorm{(\mathbf{U} - \Vh,0)} + \tnorm{(\Vh - \Uh,\Zh)}.
    \end{equation*}
    For the second term, we consider
    \begin{align*}
        B[&(\Uh - \Vh, \Zh),( \Wh, \Yh)] \\
        &= B[(\Uh,\Zh),( \Wh, \Yh)] - (\underline{v}_1, \underline{w}_1)_{\omega_T} - S_h(\Vh,\Wh) - \Sud(\Vh,\Wh) - A[\Vh,\Yh] \\
        &= (u_{\omega}, \underline{w}_1)_{\omega_T} - (\underline{v}_1, \underline{w}_1)_{\omega_T} - S_h(\Vh,\Wh) - \Sud(\Vh,\Wh) - A[\Vh,\Yh] \\
        &= (u - \underline{v}_1, \underline{w}_1)_{\omega_T}  +  S_h(\mathbf{U} - \Vh,\Wh) + \Sud(\mathbf{U} - \Vh,\Wh) + A[\mathbf{U} - \Vh,\Yh],
    \end{align*} 
    where we use the fact that $\mathbf{U}$ is sufficiently smooth and $u = u_{\omega}$ on $\omega_T$. Then, we can apply \cref{lem:continuityA} to obtain 
    \begin{align*}
        (u - &\underline{v}_1, \underline{w}_1)_{\omega_T}  +  S_h(\mathbf{U} - \Vh,\Wh) + \Sud(\mathbf{U} - \Vh,\Wh) + A[\mathbf{U} - \Vh,\Yh] \\
        &\le C(\tnorm{(\mathbf{U} - \Vh,0)}_{\wop} \tnorm{(0,\Yh)}  + \tnorm{(\mathbf{U} - \Vh,0)} \tnorm{(\Wh,0)} ) \\
        &\le C \tnorm{(\mathbf{U} - \Vh,0)}_{\wop} \tnorm{(\Wh,\Yh)}.
    \end{align*}
    Thus, the inf-sup condition on $B$ yields that 
    \begin{equation*}
    \tnorm{(\mathbf{U} - \Uh,\Zh)} \le \frac{C}{C_B} \tnorm{(\mathbf{U} - \Vh, 0)}_{\wop},
    \end{equation*}
    which gives the claim. 
\end{proof}

\begin{cor}\label{cor:tnormConvRates }
    Under the assumption of \cref{thm:bestapprox} and assuming additionally that $\Delta t = C' h$ for some $C'>0$ we have that
    \begin{equation*}
	    \tnorm{(\mathbf{U} - \Uh,\Zh)} \le C h^s \Vert u \Vert_{H^{s+2}(Q)},
    \end{equation*}
    where $s := \min\{ k,q \}$.
\end{cor}

\begin{proof}
    Due to the assumption that $c_i \in C^\infty(\Omega_i)$, $i = 1,2$, there exists $C > 0$ s.t.  
    \begin{equation*}
        \tnorm{(\Uh - \Vh, 0)}_{\wop} \le C \tnorm{(\Uh - \Vh, 0)}_{\square_1},
    \end{equation*}
where $\tnorm{(\cdot, \cdot)}_{\square_1}$ is defined by \eqref{eq:tnormwop} with $c_i = 1$, $i = 1,2$. Thus, we can apply the interpolation results derived in \cite[Lemma 5 \& 11]{BP24} to obtain the claim. 
\end{proof}

If the problem is Lipschitz stable, for instance in the homogeneous case (see \cref{thm:Lipschitz}), we can prove convergence in a physically relevant norm. To deal with the discontinuity in time of the discrete solution, a suitable lifting operator $L_{\Delta t} : \FullyDiscrSpace{k}{q} \rightarrow C^0(0,T;V_h^k)$ that ensures that $L_{\Delta t} \ul_1 \in H^1(Q)$ has been introduced in \cite[Sec. 4]{BP24}. Then, we state the following result.

\begin{thm}\label{thm:convergence}
    Let \cref{assum:LipschitzStability} be satisfied and assume that $\Delta t = C' h$. For $u$ being a sufficiently regular solution of \eqref{eq:waveEquation} and $(\Uh,\Zh) \in \ProdFullyDiscrSpace{k}{q} \times \ProdFullyDiscrSpace{k^\ast}{q^\ast}$ being the solution to \eqref{eq:discreteProblem}, there exists a constant $C>0$ such that
    \begin{equation}
	    \Vert u - L_{\Delta t} \ul_1 \Vert_{L^\infty(0,T;L^2(\Omega))} + \Vert \dt (u - L_{\Delta t} \ul_1) \Vert_{L^2(0,T;H^{-1}(\Omega))} \le C h^s \Vert u \Vert_{H^{s+2}(Q)}, 
    \end{equation}
    where $s := \min\{ k,q \}$. In particular, we have that 
    \begin{equation}
	    \Vert u_0 - \ul_1 \vert_{t = 0} \Vert_{L^2(\Omega)} + \Vert u_1 - (\dt \ul_1) \vert_{t = 0} \Vert_{H^{-1}(\Omega)} \le C h^s \Vert u \Vert_{H^{s+2}(Q)}.
    \end{equation}
\end{thm}

\begin{proof}
    The proof follows with the same arguments as the proof of \cite[Theorem 14]{BP24}.
\end{proof}

\begin{rmk}[Homogenous media]\label{rem:homMedia}
    If $c_i = 1$ for $i = 1,2$, \cref{assum:LipschitzStability} holds (see \cref{thm:Lipschitz}), and we recover the result from \cite[Theorem 14]{BP24}.
\end{rmk}

\begin{rmk}[Logarithmic stability]
    If \cref{assum:LipschitzStability} is not satisfied, we could still hope to utilize the results from \cref{thm:filippas}, provided that $T$ is large enough such that \eqref{eq:Tcondition} holds. To this end, we would have to deal with the following technicalities:  
    \begin{enumerate}
        \item The results from \cref{thm:filippas} are formulated in terms of $\Vert \wop u \Vert_{L^2(Q)}$, but we can only control $\Vert \wop \tilde{e}_h \Vert_{H^{-1}(Q)}$ for the residual $\tilde{e}_h = u - L_{\Delta t} \ul_1$ to the right order. Thus, we would have to rewrite the estimates from \cref{thm:filippas} with $\Vert \wop u \Vert_{H^{-1}(Q)}$ instead of $\Vert \wop u \Vert_{L^2(Q)}$. To this end, one might be able to shift the regularity suitably, cf. \cite[Sec. 2.6]{StolkPhD}, or use similar arguments as in \cite[Thm. A.1]{BFMO21}.
        \item We have to ensure that the constants in the estimates from \cref{thm:filippas} remain bounded when we consider the residual $\tilde{e}_h$. In particular, we have to control $C_{0,1} := \Vert (\tilde{e}_h \vert_{t = 0}, \dt \tilde{e}_h \vert_{t = 0}) \Vert_{H^1 \times L^2}$. We note that we can always ensure that $C_{0,1}$ is bounded by adding a suitable Tikhonov regularization term to \eqref{eq:primalStab}. 
    \end{enumerate}
    Then, we can use the arguments of \cite[Theorem 14]{BP24} to show that 
    \begin{equation*}
	    \Vert (u - \ul_1) \vert_{t = 0} \Vert_{L^2(\Omega)}  + \Vert \dt (u - \ul_1) \vert_{t = 0} \Vert_{H^{-1}(\Omega)} \le C C_{0,1} \log(1 + C_{0,1} h^{-s} \Vert u \Vert_{H^{s+2}(Q)}^{-1})^{-1}. 
    \end{equation*}
\end{rmk}

\section{Numerical experiments}
In this section, we present numerical experiments carried out with the proposed method. The examples in \Cref{ssection:1d-numexp}-\Cref{ssection:2d-numexp} are implemented using the \texttt{FEniCSx} library \cite{BarattaEtal2023,BasixJoss,ScroggsEtal2022,AlnaesEtal2014}.
\jp{The example in \Cref{ssection:whispering-gallery} uses the finite element software \texttt{Netgen/NGSolve} \cite{JS97, JS14} and in addition the space-time functionality provided by the add-on \texttt{ngsxfem} \cite{ngsxfem2021}.}  
\jp{A reference to reproduction material can be found in the section on \textbf{Code availability}.}

\subsection{Examples in one space dimension}\label{ssection:1d-numexp}
 We consider the unit interval $\Omega \coloneqq (0,1) \subset \mathbb{R}$ partitioned into $2^{L+1}$, $L \in \mathbb{N}$, uniform elements of size $h = 1/2^{L+1}$. Unless specified otherwise, we choose the time step as $\Delta t = h$. \changed{Further, we set $q = k$ and choose $k^{\ast} = 1$, $q^{\ast} = 1$ if $k = 1$ and $q^{\ast} = 0$ if $k > 1$.}

\subsubsection{A first example}\label{sec:numex:1D:simple}
 We define two subdomains $\Omega_1 = (0,0.5)$ and $\Omega_2 = (0.5,1.0)$ such that $\bar{\Omega} = \overline{\Omega}_1 \cup \overline{\Omega}_2$ and fix the wavespeed $c_2 := c \vert_{\Omega_2} = 1$ in $\Omega_2$. For variable $c_1 := c \vert_{\Omega_1}$, we consider the exact solution \cite{MHI08}
\begin{equation}\label{eq:1D:exact:simple}
    u(x,t) := \begin{cases}
        \cos(w_1 c_1 t) \cos(w_1(x-0.5)), & x \in \Omega_1, \\
        \cos(w_2 c_2 t) \cos(w_2(x-0.5)), & x \in \Omega_2,
    \end{cases}
\end{equation}
where we set $w_1 = 3 \pi$ and $w_2 = w_1 c_1 / c_2$. Furthermore, we define
\begin{equation}
	\omega \coloneqq [0.0,0.25] \cup [0.75,1.0], \quad \text{and} \quad \Omega_R \coloneqq \bar{\Omega} \setminus \omega = [0.25,0.75].
\end{equation}

In one space dimension, the requirement \eqref{eq:Tcondition} that $T > 2 \sup_{x \in \Omega} \operatorname{dist}_{c}(x,\omega_T)$ already implies the GCC. \changed{Thus, we can focus for the one-dimensional experiments on the condition \eqref{eq:Tcondition}.}
Let $d(x,y) \coloneqq \vert x - y \vert$, $x, y \in \Omega$, be the standard Euclidean metric. As in \cite{Filippas22}, we define a metric scaled with the discontinuous wave speed through
\begin{equation}
    d_c(x,y) \coloneqq \mathbbm{1}_{\Omega_1} c_1(x)^{-1} d(x,y) + \mathbbm{1}_{\Omega_2} d(x,y), \quad x,y \in \Omega.
\end{equation}
Thus, for $\omega$ as above, $c_2 = 1$ fixed and $c_1 > 1.0$, the condition \eqref{eq:Tcondition} translates to $T > 0.25(1+c_1^{-1})$. First, we consider the contrast $c_1 = 2.5$, which means that the GCC is satisfied if $T > 0.35$.  

In \Cref{fig:jumpCoefs:contrast:2.5}, we observe that the errors $\Vert u-L_{\Delta t} \ul_1 \Vert_{L^\infty(0,T;L^2(\Omega_R))}$ and $\Vert \partial_t (u-L_{\Delta t} \ul_1) \Vert_{L^2(0,T;L^2(\Omega_R))}$ converge with optimal order $\mathcal{O}(h^k)$ for $T = 0.5 > 0.35$. However, if $T = 0.1 < 0.35$, the order of convergence degrades, while the respective $L^2$-best approximation\textsuperscript{3}\footnotemark \footnotetext{\textsuperscript{3}\changed{We compare with the $L^2$-best approximation error, where we solve for the $L^2$-orthogonal projection of the exact solution onto the discrete space, as a measure of the approximation quality of the discrete space.}} errors still converge with optimal order. This suggests that the GCC seems to imply Lipschitz stability even in the heterogeneous case, and therefore \cref{assum:LipschitzStability} is indeed justified.

To investigate the dependence of the errors on the contrast, we fix the polynomial degree $k = 3$, refinement level $L = 3$, and $c_2 = 1$, and consider increasing values of $c_1 \in \{1.0,1.5,\dots,4.5\}$. In \cref{fig:jumpCoefs:increasingcontrast}, we observe that the approximation errors increase polynomially with $\mathcal{O}(c_1^3)$, while the $L^2$-best approximation errors grow quadratically with $\mathcal{O}(c_1^2)$. This is an important observation since an exponential growth in the contrast would have rendered the corresponding problems nearly impossible to resolve numerically.

However, this also entails that a strong refinement of the discretization is required to maintain the accuracy of the numerical solution as the contrast increases. We illustrate this problem in \cref{fig:jumpCoefs:quality}. For $c_1 = 2.5$, the exact solution is already well-captured by the numerical solution on refinement level $L=2$. However, for higher contrast $c_1 = 5.5$ this refinement level is clearly insufficient.

\begin{figure}[!htbp]
    \centering 
    \includegraphics[width=0.97\textwidth]{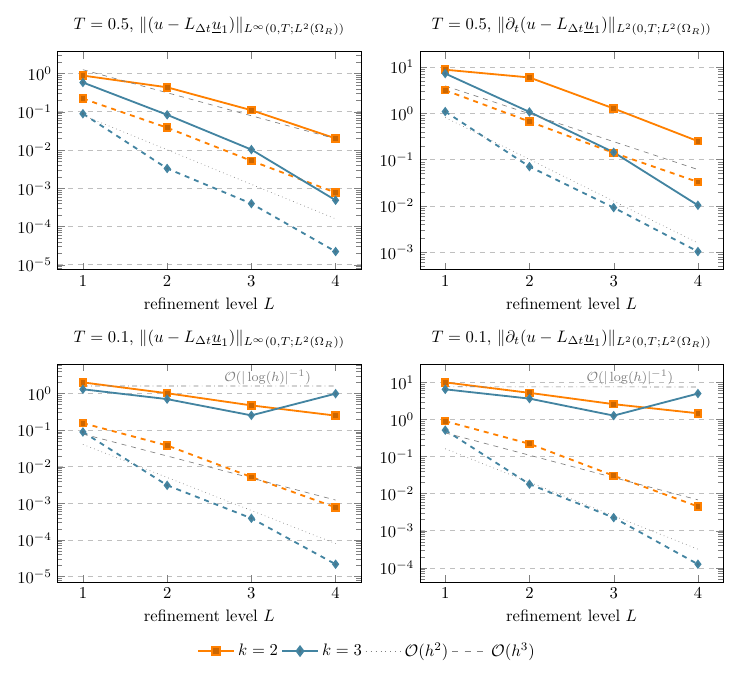}
    \caption{For the approximation of \eqref{eq:1D:exact:simple} with polynomial degree $k \in \{2,3\}$ and contrast $c_1 = 2.5$ we measure the errors $\Vert u-L_{\Delta t} \ul_1 \Vert_{L^\infty(0,T;L^2(\Omega_R))}$ (left) and $\Vert \partial_t (u-L_{\Delta t} \ul_1) \Vert_{L^2(0,T;L^2(\Omega_R))}$ (right) and compare to the respective error of the $L^2$-best approximation (dashed). For final time $T = 0.5 > 0.35$ (top), we observe optimal convergence rates, while for $T = 0.1 < 0.35$ (bottom) the order of convergence degrades.}
    \label{fig:jumpCoefs:contrast:2.5}
  \end{figure}

  \begin{figure}[!htbp]
    \centering
    \includegraphics[width=0.97\textwidth]{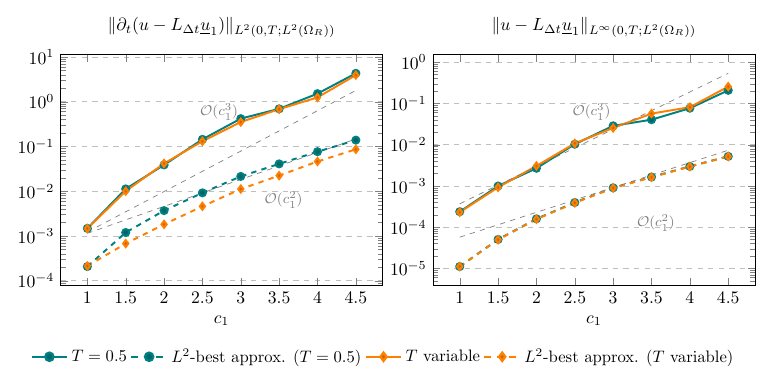}
    \caption{For $k = 3$ and $L = 3$, we consider the errors $\Vert \dt (u-L_{\Delta t} \ul_1) \Vert_{L^2(0,T;L^2(\Omega_R))}$ (left) and $\Vert u-L_{\Delta t} \ul_1 \Vert_{L^\infty(0,T;L^2(\Omega_R))}$ (right) and the respective $L^2$-best approximation errors (dashed) for increasing contrast $c_1 \in \{1.0,1.5,...,4.5\}$. We compare the case where $T = 0.5$ is fixed for all choices of $c_1$ and the case where $T$ adapted such that it barely fulfills $T > 0.25(1+c_1^{-1})$.}
    \label{fig:jumpCoefs:increasingcontrast}
  \end{figure}

\begin{figure}[!htbp]
    \centering
    \includegraphics[width=0.97\textwidth]{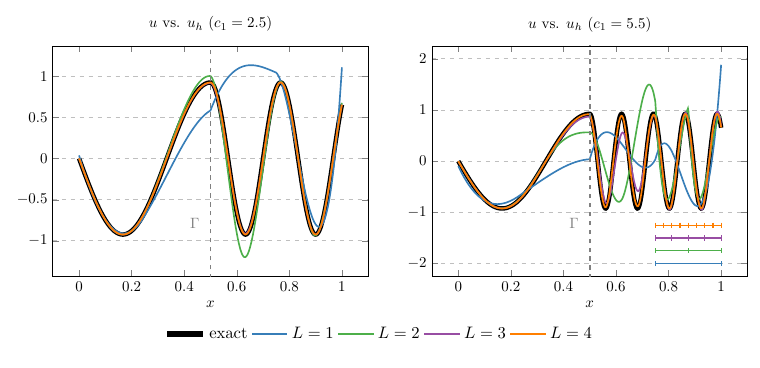}
    \caption{At time $t = 0.25$, we compare the exact solution \eqref{eq:1D:exact:simple} and approximations with polynomial degree $k = 3$ on refinement levels $L \in \{1,2,3,4 \}$ for contrasts $c_1 = 2.5$ (left) and $c_1= 5.5$ (right). The scale in the right plot indicates the mesh size $h = 1/2^{L+1}$ for the different refinement levels.}
    \label{fig:jumpCoefs:quality}
\end{figure}

\subsubsection{Multiple jumps}
To increase the complexity of the problem, we modify the exact solution \eqref{eq:1D:exact:simple} to allow for multiple jumps in the wave speed $c$. For two points $p_1,p_2 \in (0,1)$, we decompose $\Omega$ into three subdomains $\Omega_1 = (0,p_1)$, $\Omega_2 = (p_1,p_2)$, and $\Omega_3 = (p_2,1)$ s.t. $\bar{\Omega} = \cup_{i = 1,2,3} \overline{\Omega}_i$. Then, we consider the following ansatz for the exact solution: 
\begin{equation}\label{eq:1D:exact:simpleMult}
    u(x,t) := \begin{cases}
        \cos(w_1 c_1 t) \cos(w_1(x-p_1)), & x \in \Omega_1, \\
        \cos(w_2 c_2 t) \cos(w_2(x-p_1)), & x \in \Omega_2, \\
        \cos(w_3 c_3 t) \cos(w_3(x-p_2)), & x \in \Omega_3. 
    \end{cases}
\end{equation}
To ensure the continuity of \eqref{eq:1D:exact:simpleMult}, we choose $p_2$ in dependence of $p_1$ and $w_2$. For $n \in \mathbb{Z}$, we set
\begin{align*}
    p_2 \coloneqq \frac{2 \pi n + w_2 p_1}{w_2} \quad \Longrightarrow \quad \cos(w_2(p_2-p_1)) = 1.
\end{align*}
We fix $c_2 = 1.0$ and set $w_1 = 3 \pi$, $w_2 = w_1 c_1/c_2$,$c_3 = c_1$, and $w_3 = w_1$.

As a first example, we consider the data domain $\omega = [0.0,0.3] \times [0.7,1.0]$ and choose $p_1 = 0.4$ and $n = 3$ such that $p_2 \approx 0.667$. Thus, the interval $(p_1,p_2)$ is not contained in the data domain. 

For contrast $c_1 = c_3 \in \{ 7.5,11.5 \}$, \cref{fig:jumpCoefs:multipleJumps} shows the exact solution and the approximation with $k = 3$ for fixed time step $\Delta t = 1/32$ on the refinement levels $L \in \{3,4\}$. In the case of contrast $c_1 = 11.5$, the refinement level $L = 3$ does not seem to be sufficient to resolve the solution accurately, but for $L = 4$ the solution is captured reasonably well, even though we do not prescribe data in the interval $(p_1,p_2)$.  

\begin{figure}[!htbp]
    \centering
    \includegraphics[width=0.97\textwidth]{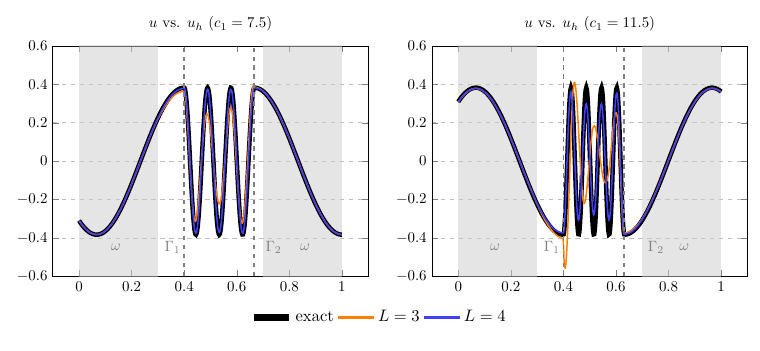}
    \caption{Exact solution \eqref{eq:1D:exact:simpleMult} and approximated solution with $k = 3$ of \eqref{eq:1D:exact:simpleMult} at $t = 0.25$ with contrast $c_1 = 7.5$ (left) and $c_1 = 11.5$ (right). We compare the approximations on the refinement levels $L = 3$ and $L = 4$, and set $\Delta t = 1/32$ for both cases.}
    \label{fig:jumpCoefs:multipleJumps}
  \end{figure}

As a second example, we consider the case where the data is only given in one half of the interval, i.e. $\omega = [0,0.3]$. The requirement \eqref{eq:Tcondition} and thus the GCC is fulfilled if 
\begin{equation}
    T > 2 \left[ (p_2 - p_1) + \frac{p_1 - 0.3 + 1 - p_2}{c_1} \right]. 
\end{equation}
We consider $c_1 \in \{ 2.5, 7.5\}$, for which the requirement $T > \{ 0.89,0.65 \}$ ensures the GCC. \Cref{fig:jumpCoefs:multipleJumps:noGCC} compares the results for polynomial degree $k = 3$ on the refinement level $L = 4$ with a fixed time step $\Delta t = 1/32$, for two final times $T = 0.5$ and $T = 1.0$. \changed{We compare the approximation of the exact solution for both final times at time $t = 0.5$}.
For both contrasts, the approximate solution matches the exact solution well if the GCC is fulfilled ($T = 1.0$), but shows significant deviations when it is not ($T = 0.5$). 
The deviation is more extreme for contrast $c_1 = 2.5$, where the gap between the considered time $T = 0.5$ and the threshold $T > 0.89$ is larger.

\begin{figure}[!htbp]
    \centering
    \includegraphics[width=0.97\textwidth]{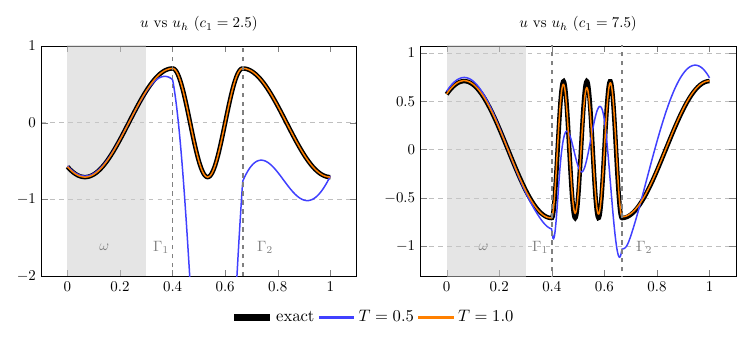}
    \caption{For fixed time step $\Delta t = 1/32$, refinement level $L = 4$, and polynomial degree $k = 3$, we compare the exact solution \eqref{eq:1D:exact:simpleMult} and approximated solution computed with \changed{final time} $T = 1.0$ and $T = 0.5$ at time $t = 0.5$. We set $c_2 = 1.0$ and consider the cases $c_1 = c_3 = 2.5$, $n = 1$ (left) and $c_1 = c_3 = 7.5$, $n = 3$ (right).} 
    \label{fig:jumpCoefs:multipleJumps:noGCC}
\end{figure}

\subsection{Example in two space dimensions}\label{ssection:2d-numexp}
We consider another example in two space dimensions. We partition the unit-square $\Omega \coloneqq [0,1]^2 \subset \mathbb{R}^2$ into subdomains $\Omega_1 = \{ (x,y) \in \Omega: x < 0.5 \}$ and $\Omega_2 = \{ (x,y) \in \Omega: x > 0.5 \}$, and 
extend the exact solution defined in \eqref{eq:1D:exact:simple} to the two-dimensional case by setting $u(x,y;t) \coloneqq u(x;t)$. 

We define the data domain $\omega \coloneqq \Omega \setminus [0.25,0.75]^2$ and note that $\omega$ fulfills the GCC. To ensure that the data domain and the interface $\Gamma$ are meshed exactly, we consider a quadrilateral mesh with $2^{L+1}$, $L \in \{1,2,3,4\}$, elements in each direction. We consider the final time $T = 0.75$  and set the time step size to $\Delta t = (1/2)^{L+1}$. As before, we fix the wave speed $c_2 = 1$ and consider different values for $c_1$ to control the contrast. \changed{Further, we choose the parameters $q = k^{\ast} = q^{\ast} = k$.}

The results of the experiment are presented in \cref{fig:2DjumpCoefs}. 
With contrast $c_1 = 2.5$, we observe quasi-optimal convergence with order  $\mathcal{O}(h^k)$ for polynomial degrees $k \in \{2,3\}$ for the errors $\Vert \dt(u-L_{\Delta t} \ul_1) \Vert_{L^2(0,T;L^2(\Omega))}$ and $\Vert u-L_{\Delta t} \ul_1 \Vert_{L^\infty(0,T;L^2(\Omega))}$. 
Similar as already observed in the one-dimensional case in \cref{fig:jumpCoefs:increasingcontrast}, we observe that these errors are more sensitive to increasing contrast $c_1$ than the corresponding $L^2$-best approximation errors. 
In particular, the ratio for $\Vert u-L_{\Delta t} \ul_1 \Vert_{L^\infty(0,T;L^2(\Omega))}$ with $k=2$ appears to grow as $\mathcal{O}(c_1^{3/2})$, which renders it very demanding to maintain a reasonable accuracy as the contrast increases. 

\begin{figure}[!htbp]
    \centering
    \includegraphics[width=0.97\textwidth]{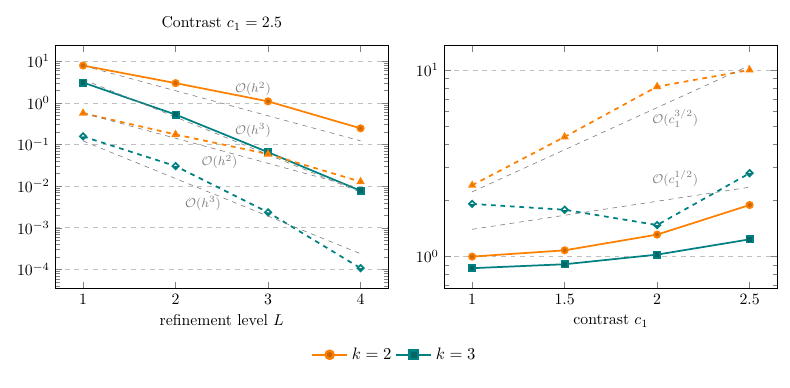}
    \caption{On the left: We observe quasi-optimal convergence of $\Vert \dt(u-L_{\Delta t} \ul_1) \Vert_{L^2(0,T;L^2(\Omega))}$ (solid) and $\Vert u-L_{\Delta t} \ul_1 \Vert_{L^\infty(0,T;L^2(\Omega))}$ (dashed) with contrast $c_1 = 2.5$ and polynomial degrees $k \in \{2,3\}$. On the right: We consider the ratio between these errors and the $L^2$-best approximation error for increasing contrast $c_1$.}
    \label{fig:2DjumpCoefs}
\end{figure}

\subsection{Whispering gallery modes in an annulus}\label{ssection:whispering-gallery}
\jp{
As a final example, we would like to consider a two-dimensional setting in which the Lipschitz stability estimates from \cref{assum:LipschitzStability} are not valid. 
To this end, we follow reference \cite{Filippas23}, in which it has been shown that, under certain assumptions on the contrast of the sound speed and the geometry, so-called whispering gallery modes exist. These are solutions of the wave equation that are exponentially decaying away from the interface, showing that the logarithmic stability estimate \eqref{eq:FilippasEstimateLog} from \cite{Filippas22} is in general optimal.

Indeed, in certain domains $\Omega$, we can follow \cite{Filippas23} to construct a sequence of eigenfunctions $w_{\ell}$  of $-\Delta_c : = - \div(c^2 \nabla  \cdot)$ with corresponding eigenvalues $\lambda_{\ell} \rightarrow \infty$, such that for all $\omega \subset \Omega$ with $\dis(\bar{\omega},\Gamma) > 0$, there exists constants $C,D > 0$ such that  
\[
\norm{ w_{\ell} }_{ L^2(\omega) } \leq C \exp(-D \sqrt{ \lambda_{\ell } }  ), \quad 
- \Delta_c w_{\ell } = \lambda_{\ell} w_{\ell }, \quad 
w_{\ell}|_{\partial \Omega } = 0, \quad
\norm{ w_{\ell }  }_{ L^2( \Omega) } = 1.  
\]
Then, we obtain a family of solutions of the wave equation
\bel{eq:whispering-gallery-mode}
u_{\ell}(t,x) = \cos( \sqrt{ \lambda_{\ell} } t ) w_{\ell}(x),   \quad \norm{ u_{\ell}|_{t=0}}_{ L^2(\Omega) }  + \norm{ \partial_t u_{\ell}|_{t=0} }_{ H^{-1}( \Omega )}  = 1,
\ee
whose initial condition can not be controlled uniformly in a Lipschitz stable manner by the data. This follows, since 
\[
\Vert  u_{\ell} \Vert_{L^2(\omega_T)} + \Vert u_{\ell} \Vert_{L^2(\Sigma)} + \Vert \square_c u_{\ell} \Vert_{H^{-1}(Q)}  \leq C_T \exp(-D \sqrt{ \lambda_{\ell } }  )
\]
goes to zero as $\ell \rightarrow \infty$.

As shown in \cite[Section 2]{Filippas23}, one possible geometrical configuration to accommodate such modes is an annulus, where $\Omega_1 = \{  R_0 < r < R_1 \} $ and $\Omega_2 = \{  R_1 < r < R_2 \} $ with 
$R_0 < R_1 < R_2$. To ensure the existence of whispering Gallery modes the contrast has to fulfill the condition 
\[
\frac{c_1}{R_1} < \frac{c_2}{R_2}.
\]
Due to the spherical geometry, we can obtain the mode $w_{\ell}$ by separation of variables
\bel{eq:whispering-sep}
w_{\ell}(r,\theta) = \exp(i \ell \theta) \psi_{\ell}(r)
\ee
and by solving the one-dimensional eigenvalue problem
\bel{eq:eigval-prob-r}
- \frac{c^2}{\ell^2} \left[  \partial_r^2 \psi_{\ell}  + \frac{1}{r} \partial_r \psi_{\ell} \right] + \frac{c^2}{r^2} \psi_{\ell} = \frac{ \lambda_{\ell} }{ \ell^2 } \psi_{ \ell}
\ee
with Dirichlet boundary conditions $\psi_{\ell}(r=R_0) = \psi_{\ell}(r=R_2) = 0$. 

While there exists many eigenpairs $( \lambda_{\ell}, \psi_{\ell} )$ for fixed $\ell$ that solve \eqref{eq:eigval-prob-r}, for our purposes it suffices to pick just one. We fix the parameters 
\[
R_0 = 3/5, \; R_1 = 1, \; R_2 = 3/2, \quad c_1 = 1, \; c_2 = \sqrt{20}
\]
and find two eigenpairs for $\ell =1$ and $\ell=10$ whose (radial) eigenfunctions $\psi_{\ell}$ are plotted in the upper panel of \cref{fig:whispering-gallery}.
The corresponding solutions of the wave equation}\textsuperscript{4}\footnotemark \ \footnotetext{\textsuperscript{3}{\jp{In the numerical experiments we work only with the real part of $w_{\ell}$ to obtain real-valued solutions}}}\jp{$u_{\ell}(t,x)$ as defined in \cref{eq:whispering-gallery-mode} are also displayed at the final time $t=T$ for 
$T=1$. 
The corresponding eigenvalues are $\lambda_{1} \approx 190.669$ and $\lambda_{10} \approx 404.347$. Since $\lambda_{10} >  \lambda_{1}$, the corresponding eigenfunction decays much faster towards the right of the interface at $r=1$ as can be seen in the central panel of \cref{fig:whispering-gallery}. The eigenfunctions shown here do not fall off significantly towards the left of the interface since we consider very low $\ell$ in order to keep the computational costs moderate. For higher $\ell \approx 200$, the corresponding eigenfunctions indeed decay rapidly also towards the left of the interface. However, these functions oscillate heavily in $\theta$-direction and in time making
them very challenging to resolve numerically. The two modes we have chosen for $\ell=1$ and $\ell=10$ are fortunately sufficient for our purpose. 
}

\begin{figure}[!htbp]
   \begin{center}
    \includegraphics[width=0.98\textwidth]{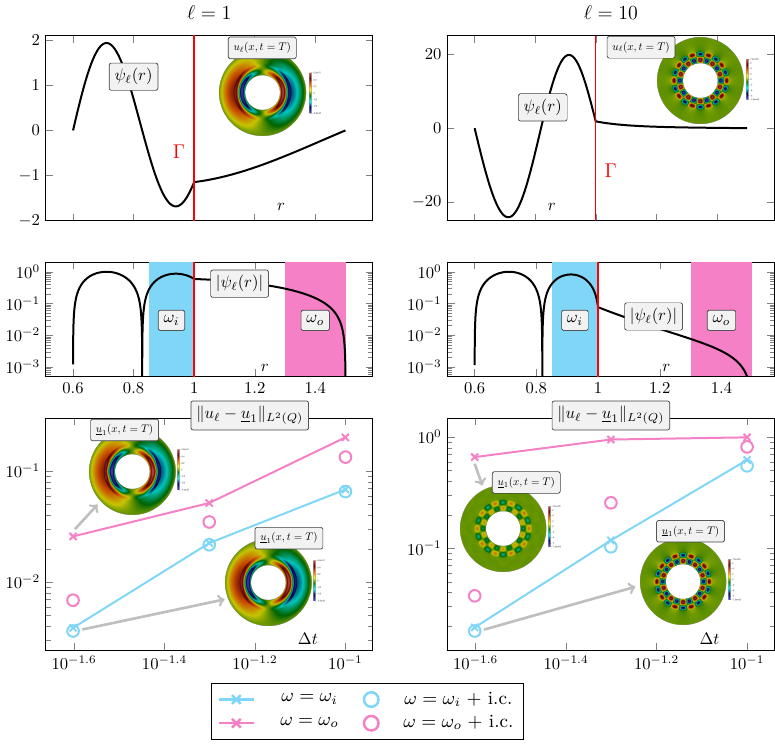}
   \end{center}
	\caption{ \jp{ The upper panel displays the modes $\psi(r)$ for $\ell=1$ (left), respectively $\ell=10$ (right) obtained by solving the eigenvalue problem \cref{eq:eigval-prob-r}. 
	 The insets show the corresponding solutions of the wave equation defined by \cref{eq:whispering-gallery-mode} and \cref{eq:whispering-sep}  at $t=T$. 
         The absolute value of $\psi(r)$ is shown in the central panel. The lower panel displays the convergence of $  \Vert u_{\ell}  - \underline{u}_1 \Vert_{ L^2(Q) } $.
	 The case where the data is given in $\omega_i$ is shown in magenta (solid lines with crosses) whereas the case where data is given in $\omega_o$ is displayed in cyan (same linestyle).
	 For reference, the circles (marked as '+ i.c.' in the legend) display the errors when the method is run with initial conditions $u_{\ell}(\cdot, t=0)$ and $\partial_t u_{\ell}(\cdot, t=0)$ 
	 known everywhere. The insets in the lower panel display the numerical reconstructions achieved at the final time corresponding to the refinement level indicated by the arrows. }
	}
    \label{fig:whispering-gallery}
\end{figure}

\jp{
We want to evaluate the performance of our method in reconstructing the modes $u_{\ell}$, particularly, in the case where the data domain $\omega$ is far away from the interface so that the mode for $\ell=10$ is very small in $\omega$. For this purpose, we consider the data domain $\omega_o  = \{  1.3 < r < R_2 \} $ in which the mode is very small according to the central panel of \cref{fig:whispering-gallery}. For comparison, we also consider a data domain close to the interface $\omega_i  = \{  0.85 < r < R_1 \}$.

The $L^2$-errors in $Q = (0,T) \times \Omega$ for the unique continuation problem of the mode $u_{\ell}$ computed with polynomial degree $q = k^{\ast} = q^{\ast} = k = 2$ are displayed in the lower panel of \cref{fig:whispering-gallery}.
Here, we have chosen $T=1$ such that the condition \cref{eq:Tcondition} for unique solvability is fulfilled.
When the data is given in $\omega_i$ (i.e. close to the interface), then the method converges with the same order as when the initial condition is given, i.e. the convergence is as good as for a well-posed problem irrespective of the considered mode $\ell$. However, when the data is given far away from the interface in $\omega_0$, then 
the rate of convergence drastically deteriorates as $\ell$ increases. 
This is in full agreement with the discussion in the beginning of this subsection: the Lipschitz stability result of \cref{assum:LipschitzStability} cannot hold for $\dis(\bar{\omega}_0,\Gamma) > 0$. The stability is of logarithmic type as asserted by  \cref{thm:filippas} (Thm. 5.19 of \cite{Filippas22}) and this is exemplified in our numerical experiments by the observation that the convergence rate appears to drops to logarithmic type as $\ell$ increases.
}

\section{Conclusion}
In this article, we investigated a numerical method for solving the variational data assimilation problem for the wave equation with discontinuous wave speed. Our numerical experiments lead to two main conclusions. First, obtaining accurate numerical solutions in the presence of material interfaces involving large contrasts is computationally demanding, and the resolution required to capture the complex behavior of the solution is a limiting factor in numerical simulations. This challenge is amplified by the fact that unique continuation problems are globally coupled in space and time. Second, even in the case of a discontinuous wave speed, the geometric control condition seems to ensure Lipschitz stability, allowing the design of accurate and reliable numerical methods. 

\bmhead{Acknowledgments}
The authors would like to thank Spyridon Filippas and Lauri Oksanen for interesting discussions.

\section*{Declarations}
%
On behalf of all authors, the corresponding author states that there is no conflict of interest.
\begin{itemize}
\item \textbf{Funding} Funding by EPSRC grant EP/V050400/1 is gratefully acknowledged.
\item \textbf{Code availability} Reproduction material is available at \url{https://github.com/TimvanBeeck/waveUC_discCoefs}.
\end{itemize}

\bibliography{references}

\end{document}